\documentclass[a4paper,pdftex,reqno,11pt]{amsart}

\allowdisplaybreaks[1]

\usepackage{geometry}

\usepackage{graphicx}
\usepackage{enumerate}
\usepackage{courier}
\usepackage{times}
\usepackage{amsmath,amssymb}
\usepackage{algorithm}
\usepackage{algorithmic}
\usepackage{hyperref}
\usepackage{color}

\theoremstyle{plain}
\newtheorem{Theorem}{Theorem}[section]
\newtheorem{Proposition}[Theorem]{Proposition}
\newtheorem{Conjecture}[Theorem]{Conjecture}
\newtheorem{Corollary}[Theorem]{Corollary}
\newtheorem{Lemma}[Theorem]{Lemma}

\newenvironment{Proof}
{\begin{trivlist}\item[]{{\sc Proof.}}}{\hfill{$\square$}\noindent\end{trivlist}}

\theoremstyle{definition}
\newtheorem{Definition}[Theorem]{Definition}

\theoremstyle{remark}

\DeclareMathOperator{\rk}{rk}
\newcommand{\gaussmnum}[3]{\left[\begin{smallmatrix}{#1}\\{#2}\end{smallmatrix}\right]_{#3}}
\newcommand{\gaussmset}[2]{\left[\begin{smallmatrix}{#1}\\{#2}\end{smallmatrix}\right]}
\newcommand{\F}{\ensuremath{\mathbb{F}}}
\newcommand{\Fq}{\ensuremath{\F_q}}
\newcommand{\Fqn}{\ensuremath{V}}
\newcommand{\PFqn}{\ensuremath{P(\Fqn)}}
\newcommand{\Gqnk}{\ensuremath{\gaussmset{\Fqn}{k}}}
\newcommand{\dham}{\mathrm{d}_{\mathrm{h}}}
\newcommand{\rdist}{\mathrm{d}_{\mathrm{r}}}
\newcommand{\sdist}{\mathrm{d}_{\mathrm{s}}}

\newcommand{\cdc}{\texttt{cdc}}
\newcommand{\cdcs}{\texttt{cdcs}}
\newcommand{\rref}{\texttt{rref}}
\newcommand{\MRD}{\texttt{MRD}}

\begin{document}


\title{A note on the linkage construction for constant dimension codes}


 \author{Sascha Kurz}
 \address{Sascha Kurz, University of Bayreuth, 95440 Bayreuth, Germany}
 \email{sascha.kurz@uni-bayreuth.de}

\abstract{Constant dimension codes are used for error control in random linear network coding, 
so that constructions for these codes with large cardinality have achieved wide attention in 
the last decade. Here, we improve the so-called linkage construction. 
Known upper bounds for constant dimension codes containing lifted {\MRD} codes are generalized.
}}

\maketitle

\section{Introduction}
Let $V\cong{\F}_q^v$ be a $v$-dimensional vector space over the finite field {\Fq} with $q$ elements.
By {\Gqnk} we denote the set of all $k$-dimensional subspaces in {\Fqn}, where $0\le k\le v$. The size of the so-called 
\emph{Grassmannian} {\Gqnk} is given by the $q$-binomial coefficient $\gaussmnum{v}{k}{q}:= \prod_{i=1}^{k} \frac{q^{v-k+i}-1}{q^i-1}$. More generally,
the set ${\PFqn}$ of all subspaces of $V$ forms a metric space with respect to the subspace distance defined by
$\sdist(U,W) = \dim(U+W)-\dim(U \cap W)=\dim(U)+\dim(W)-2\dim(U \cap W)$. Coding theory on {\PFqn} is motivated by K{\"o}tter
and Kschi\-schang \cite{MR2451015} via error correcting random network coding. For $\mathcal{C}\subseteq {\Gqnk}$ 
we speak of a \emph{constant dimension code} (\cdc), where the minimum subspace distance $\sdist$ is always an even integer. By 
a $(v,N,d;k)_q$ code we denote a {\cdc} in {\Fqn} with minimum (subspace) distance $d$ and cardinality $N$. The corresponding maximum size 
is denoted by $A_q(v,d;k)$. In geometrical terms, a $(v,N,d;k)_q$ code $\mathcal{C}$ is a set of $N$ $k$-dimensional subspaces of 
$V$, $k$-spaces for short, such that any $(k-d/2+1)$-space is contained in at most one element of $\mathcal{C}$. In other words, 
each two different codewords intersect in a subspace of dimension at most $k-d/2$. For two $k$-spaces $U$ and $W$ that have an 
intersection of dimension zero, we will say that they intersect trivially or are disjoint (since they do not share a common point). 
We will call $1$-, $2$-, $3$-, and $4$-spaces, points, lines, planes, and solids, respectively. For the known lower and upper bounds 
on $A_q(v,d;k)$ we refer to the online tables \url{http://subspacecodes.uni-bayreuth.de} associated with the survey \cite{HKKW2016Tables}.
Here we improve the so-called \emph{linkage construction} \cite{MR3543532} and 
generalize upper bounds for constant dimension codes containing \emph{lifted {\MRD} codes}.

\section{Preliminaries}
In the following we will mainly consider the case $V={\F}_q^v$ in order to simplify notation. We associate with a subspace $U\in {\Gqnk}$ a 
unique $k\times v$ matrix $X_U$ in row reduced echelon 
form (\rref) having the property that $\langle X_U\rangle=U$ and denote the corresponding bijection
\[
  \gaussmset{\mathbb{F}_q^v}{k} \rightarrow \{X_U \in \mathbb{F}_q^{k \times v} | \rk(X_U)=k, \text{$X_U$ is in \rref}\}
\]
by $\tau$. An example is given by $X_U=\left(\begin{smallmatrix}1&0&0\\0&1&1\end{smallmatrix}\right)\in \mathbb{F}_2^{2 \times 3}$, where
$U=\tau^{-1}(X_U)\in \gaussmset{\mathbb{F}_2^3}{2}$ is a line that contains the three points $(1,0,0)$, $(1,1,1)$, and $(0,1,1)$. 
With this, we can express the subspace distance between two $k$-dimensional subspaces $U,W\in\Gqnk$ via the rank of a matrix:
\begin{equation}
  \label{eq_d_s_rk}
  \sdist(U,W)=2\dim(U+W)-\dim(U)-\dim(W)=2\left(\rk\!\left(\begin{smallmatrix}\tau(U)\\ \tau(W)\end{smallmatrix}\right)-k\right).
\end{equation}

By $p\colon \{M \in \mathbb{F}_q^{k \times v} | \rk(M)=k, \text{M is in \rref}\} \rightarrow \{x \in \mathbb{F}_2^v \mid \sum_{i=1}^v x_i = k\}$
we denote the pivot positions of the matrix in {\rref}. For our example $X_U$ we we have $p(X_U)=(1,1,0)$. Slightly abusing notation we
also write $p(U)$ for subspaces $U\in\Gqnk$ instead of $p(\tau(U))$. The Hamming distance $\dham(u,w)=\# \{i \mid u_i \ne w_i\}$,
for two vectors $u,w \in \mathbb{F}_2^v$, can be used to lower bound the subspace distance between two codewords. 
\begin{Lemma}\cite[Lemma~2]{MR2589964}
\label{lemma_d_s_d_h}
For two subspaces $U,W\in P(V)$ 
we have $$\sdist(U,W) \ge \dham(p(U),p(W)).$$
\end{Lemma}

For two matrices $A,B\in\mathbb{F}_q^{m\times n}$ we define the rank distance $\rdist(A,B):=\rk(A-B)$. A subset $\mathcal{M}\subseteq \mathbb{F}_q^{m\times n}$ is
called a rank metric code. 

\begin{Theorem}(see \cite{gabidulin1985theory})
  \label{thm_MRD_size}
  Let $m,n\ge d'$ be positive integers, $q$ a prime power, and $\mathcal{M}\subseteq \mathbb{F}_q^{m\times n}$ be a rank metric
  code with minimum rank distance $d'$. Then, $\# \mathcal{M}\le q^{\max\{n,m\}\cdot (\min\{n,m\}-d'+1)}$.
\end{Theorem}
Codes attaining this upper bound are called maximum rank distance (\MRD) codes. They exist for all choices of parameters. 
If $m<d'$ or $n<d'$, then only $\# \mathcal{M}=1$ is possible, which can be achieved by a zero matrix and may be summarized to the single upper bound
$\# \mathcal{M}\le \left\lceil q^{\max\{n,m\}\cdot (\min\{n,m\}-d'+1)}\right\rceil$.
Using an $m\times m$ identity matrix $I_{m\times m}$ as a prefix one obtains the so-called lifted {\MRD} codes, i.e., the {\cdc} 
$\left\{\tau^{-1}(I_{m\times m}|A) \mid A\in\mathcal{M}\right\}\subseteq \gaussmset{\F_q^{m+n}}{m}$, where $(B|A)$ denotes the concatenation of the matrices $B$ and $A$. 
\begin{Theorem}\cite[Proposition 4]{silva2008rank}
  For positive integers $k,d,v$ with $k\le v$, $d\le 2\min\{k,v-k\}$, and $d$ even, the size of a lifted {\MRD} code $\mathcal{C}\subseteq \Gqnk$ with
  minimum subspace distance $d$ is given by \[\#\mathcal{C}=M(q,k,v,d):=q^{\max\{k,v-k\}\cdot(\min\{k,v-k\}-d/2+1)}.\] If $d>2\min\{k,v-k\}$, then we set $M(q,k,v,d):=1$.
\end{Theorem}

So, for positive integers $v$, $k$, and $d$ with $d\le 2k\le v$ and $d\equiv 0\pmod 2$ we have
$$
  A_q(v,d;k)\ge q^{(v-k)\cdot(k-d/2+1)}.
$$  
For a $(v,\star,d;k)_q$ code $\mathcal{C}$ each $(k-d/2+1)$-space is contained in at most one element of $\mathcal{C}$, so that
$$
  A_q(v,d;k)\le \frac{\gaussmnum{v}{k-d/2+1}{q}}{\gaussmnum{k}{k-d/2+1}{q}},
$$
which is known as the \emph{Anticode bound} since it is based on the largest possible anticode, see e.g.\ \cite[Theorem 1]{MR867648}. 
Analyzing the right hand side we obtain
\begin{equation}
  \label{eq_asymptotic_explicit}
  q^{(v-k)\cdot(k-d/2+1)}\le A_q(v,d;k)< 2\cdot q^{(v-k)\cdot(k-d/2+1)},
\end{equation}
see e.g.\ \cite[Proposition 8]{heinlein2017asymptotic}, noting that the upper bound is also valid if $2k>v$, i.e., $k>v-k$.

We will also need to count the number of subspaces with certain intersection properties, see e.g.\ \cite[Lemma 2]{heinlein2017asymptotic}:
\begin{Lemma}
  \label{lemma_count_subspaces}
  Let $W$ be a $w$-space in $\mathbb{F}_q^v$. Then, the number of $u$-spaces $U$ in $\mathbb{F}_q^v$ with $\dim(U\cap W)=i$ is given by 
  $q^{(w-i)(u-i)}\cdot \gaussmnum{w}{i}{q}\cdot \gaussmnum{v-w}{u-i}{q}$ for all $0\le i\le \min\{u,w\}$.
\end{Lemma}

Directly from the definition of the $q$-binomial coefficients we conclude $\gaussmnum{a}{b}{q}=\gaussmnum{a}{a-b}{q}$,
\begin{equation}
  \label{eq_qbin_plus_1}
  \gaussmnum{a+1}{b}{q}/\gaussmnum{a}{b}{q}=\prod_{i=1}^{b} \frac{q^{a+1-b+i}-1}{q^{a-b+i}-1}=\frac{q^{a+1}-1}{q^{a-b+1}-1}
\end{equation}
and
\begin{equation}
  \label{eq_qbin_minus_1}
  \gaussmnum{a-1}{b}{q}/\gaussmnum{a}{b}{q}=\prod_{i=1}^{b} \frac{q^{a-1-b+i}-1}{q^{a-b+i}-1}=\frac{q^{a-b}-1}{q^{a}-1}.
\end{equation}

As shown in e.g.\ \cite[Lemma 4]{MR2451015} we have 
\begin{equation}
  \label{ie_qbin}
  q^{(a-b)b}\le \gaussmnum{a}{b}{q}\le 4q^{(a-b)b}\le q^2\cdot q^{(a-b)b}. 
\end{equation}

\section{The linkage construction revisited}
In this section we briefly review the so-called linkage construction with its different variants before we present our improvement in 
Theorem~\ref{main_thm}. The basic idea is the same as for lifted {\MRD} codes. Instead of a $k\times k$ identity matrix $I_{k\times k}$ we can also use any matrix of 
full row rank $k$ as a prefix for the matrices from a rank metric code. Let $v$, $m$, $d$, and $k$ be integers with $2\le k\le v$, $2\le d\le 2k$, and $k\le m\le v-k$. Starting 
from an $(m,N,d;k)_q$ code $\mathcal{C}$ and an {\MRD} code $\mathcal{M}$ of 
$k\times (v-m)$-matrices over $\F_q$ with rank distance $d/2$, we 
can construct a {\cdc}
$$
  \mathcal{C}'=\left\{ \tau^{-1}\left(\tau(U)|A\right) \mid U\in \mathcal{C}, A \in \mathcal{M}\right\}\subseteq \ensuremath{\gaussmset{\F_q^v}{k}}.
$$ 
This generalized lifting idea was called 
\emph{Construction D} in \cite[Theorem~37]{silberstein2015error}, cf.~\cite[Theorem~5.1]{gluesing2015cyclic}. 
For different $U,U'\in\mathcal{C}$ and different $A,A'\in \mathcal{M}$ we have
$$
  \sdist\!\Big(\tau^{-1}\!\big(\tau(U)|A\big),\tau^{-1}\!\big(\tau(U)|A'\big)\Big)\ge 
  2\Big(\rk\!\big(\tau(U)\big)-k+\rk\!\left(A-A'\right)\Big)=2\rk\!\left(A-A'\right) \ge d,
$$
$$
  \sdist\!\Big(\tau^{-1}\!\big(\tau(U)|A\big),\tau^{-1}\!\big(\tau(U')|A\big)\Big)\ge
  2\left(\rk\!\left(\begin{smallmatrix}\tau(U)\\ \tau(U')\end{smallmatrix}\right)-k\right)=\sdist(U,U')\ge d,
$$
and
$$
  \sdist\!\Big(\tau^{-1}\!\big(\tau(U)|A\big),\tau^{-1}\!\big(\tau(U')|A'\big)\Big)\ge
  2\left(\rk\!\left(\begin{smallmatrix}\tau(U)\\ \tau(U')\end{smallmatrix}\right)-k\right)=\sdist(U,U')\ge d
$$
due to Equation~(\ref{eq_d_s_rk}). Since $\mathcal{C}'$ consists of $k$-spaces and has minimum subspace distance at least $d$,  
we obtain
\begin{equation}
  A_q(v,d;k)\ge A_q(m,d;k)\cdot \left\lceil q^{(v-m)(k-d/2+1)}\right\rceil
\end{equation}
for $k\le m\le v-k$.  In terms of pivot vectors we have that the $k$ ones in $p(U)$ all are contained in the first $m$ entries for all 
$U\in\mathcal{C}'$. Geometrically, there exists a $(v-m)$-space $W\le \F_q^v$ that is disjoint to all codewords. Since $W\cong \F_q^{v-m}$ 
there exists an $(v-m,N'',d;k)_q$ code $\mathcal{C}''$ of cardinality $N''=A_q(v-m,d;k)$ that can be embedded into $W$. For all 
$U'\in\mathcal{C}'$ and all $U''\in\mathcal{C}''$ we have $\sdist(U',U'')=2k\ge d$, so that   
\begin{equation}
  \label{ie_linkage_construction}
  A_q(v,d;k)\ge A_q(m,d;k)\cdot \left\lceil q^{(v-m)(k-d/2+1)}\right\rceil+A_q(v-m,d;k)
\end{equation}
for $k\le m\le v-k$. This is called \emph{linkage construction} in \cite[Theorem~2.3]{MR3543532}, cf.~\cite[Corollary~39]{silberstein2015error}. 
However, the assumption $\dim(U'\cap U'')=0$ can be weakened if $d<2k$. 
Let $W'$ be an arbitrary $\left(v-m+k-\tfrac{d}{2}\right)$-space containing $W$ and $\mathcal{C}'''$ 
be a $(v-m+k-d/2,N''',d;k)_q$ {\cdc} embedded in $W'$. For all $U'\in\mathcal{C}'$ and all $U'''\in\mathcal{C}'''$ we have $\sdist(U',U''')=2k-2\dim(U'\cap U''')\ge 
2k-2\dim(U'\cap W')\ge d$, so that
\begin{equation}
  \label{ie_linkage_construction_improved}
  A_q(v,d;k)\ge A_q(m,d;k)\cdot \left\lceil q^{(v-m)\cdot(k-d/2+1)}\right\rceil+A_q(v-m+k-d/2,d;k)
\end{equation}
for $k\le m\le v-k$. 
This is called \emph{improved linkage construction}, see \cite[Theorem~18, Corollary 4]{heinlein2017asymptotic}. 
Interestingly enough, in more than half 
of the cases covered in \cite{HKKW2016Tables}, the best known lower bound for $A_q(v,d;k)$ is obtained via this inequality. The dimension 
of the utilized subspace $W'$ is tight in general. However, we may also consider geometrically more complicated objects than subspaces.

\begin{Definition}
  \label{def_B_q}
  Let $B_q(v_1,v_2,d;k)$ denote the maximum number of $k$-spaces in $\mathbb{F}_q^{v_1}$ with minimum subspace distance $d$ such that there 
  exists a $v_2$-space $W$ which intersects every chosen $k$-space in dimension at least $d/2$, where $0\le v_2\le v_1$.
\end{Definition}

Since the group of isometries of $\mathbb{F}_q^{v_1}$, with respect to the subspace distance, acts transitively on the set of $v_2$-spaces, we can fix an 
arbitrary $v_2$-space $W$ without changing the maximum possible number $B_q(v_1,v_2,d;k)$ of $k$-spaces. 

\begin{Theorem}
\label{main_thm}
$$
  A_q(v,d;k)\ge A_q(m,d;k)\cdot \left\lceil q^{(v-m)(k-d/2+1)}\right\rceil+B_q(v,v-m,d;k)
$$
for $k\le m\le v-k$.
\end{Theorem}
\begin{proof}
  Let $k\le m\le v-k$ be an arbitrary integer, $\mathcal{C}$ be an $\left(m,N,d;k\right)_q$ code, where $N=A_q(m,d;k)$, and $\mathcal{M}$ an {\MRD}  of $k\times (v-m)$-matrices over $\F_q$ 
  with rank distance $d/2$. With this we set $\mathcal{C}':=\left\{\tau^{-1}(\tau(U)|A) \mid U\in\mathcal{C}, A\in\mathcal{M}\right\}\subseteq \ensuremath{\gaussmset{\F_q^{v}}{k}}$, 
  i.e., we apply the lifting construction to $\mathcal{C}$. As argued before, there exists a $(v-m)$-space $W$ that is disjoint from all elements from $\mathcal{C}'$. Now let 
  $\mathcal{C}''\subseteq \ensuremath{\gaussmset{\F_q^{v}}{k}}$ be a {\cdc} with minimum subspace distance $d$ such that every codeword intersects $W$ in dimension at least $d/2$, 
  which has the maximum possible cardinality.
  
  For each $U'\in\mathcal{C}'$ and each $U''\in\mathcal{C}''$ we have $\dim(U'\cap U'')\le k-d/2$ since $\dim(U')=\dim(U'')=k$, $\dim(U'\cap W)=0$, and $\dim(U''\cap W)\ge d/2$. 
  Thus, $\sdist(U',U'')\ge d$ and $A_q(v,d;k)\ge \#\mathcal{C}'+\#\mathcal{C}''=A_q(m,d;k)\cdot \left\lceil q^{(v-m)(k-d/2+1)}\right\rceil+B_q(v,v-m,d;k)$.
\end{proof}

The determination of $B_q(v,v-m,d;k)$ or $B_q(v_1,v_2,d;k)$ is a hard problem in general. 
So, we provide several parametric examples how Theorem~\ref{main_thm} can be applied to obtain improved lower bounds for $A_q(v,d;k)$ in the next section.  

An application of the linkage construction is a lower bound for $A_q(v,4;2)$. If $v\ge 4$ we can use Inequality~(\ref{ie_linkage_construction}) with $m=2$ 
to conclude $A_q(v,4;2)\ge q^{v-2}+A_q(v-2,4;2)$. 
Since $A_q(3,4;2)=A_q(2,4;2)=1$ this gives $A_q(v,4;2)\ge q^{v-2}+q^{v-4}+\dots+q^2+q^0=\gaussmnum{v}{1}{q}/\gaussmnum{2}{1}{q}$ for even $v\ge 2$ and 
$A_q(v,4;2)\ge q^{v-2}+q^{v-4}+\dots+q^3+q^0=\gaussmnum{v}{1}{q}/\gaussmnum{2}{1}{q}-\tfrac{q^2}{q+1}$ for odd $v\ge 3$, by induction on $v$. These lower bounds are indeed 
tight, see e.g.\ \cite[Theorem 4.2]{MR0404010}. 
If $v$ is even and the maximum cardinality $A_q(v,4;2)=\gaussmnum{v}{1}{q}/\gaussmnum{2}{1}{q}$ is attained the corresponding code is called a line spread. In general 
we call a set of pairwise disjoint lines a partial line spread. If $v$ is odd and we do not fill the final plane with a single codeword, then we get a partial 
line spread of cardinality $A_q(v,4;2)-1$ that is disjoint from a fixed plane $\pi$.
   
\section{Bounds for $B_q(v_1,v_2,d;k)$}   
\label{sec_bounds}

Given Theorem~\ref{main_thm} we are naturally interested in bounds for $B_q(v_1,v_2,d;k)$. If either $v_1<k$ or $v_2<\tfrac{d}{2}$, then we trivially have 
$B_q(v_1,v_2,d;k)=0$. Similarly, if $v_1\ge k$, $v_2\ge \tfrac{d}{2}$, and $d>2k$, then we also have $B_q(v_1,v_2,d;k)=0$. We will call those parameters \emph{trivial}. By refining 
the counting of $(k-d/2+1)$-spaces contained in codewords, underlying the presented argument for the Anticode bound, we obtain: 

\begin{Lemma}
  \label{lemma_lp_upper_bound}
  As an abbreviation we set $t:=k-\tfrac{d}{2}+1$ and $b(i,j):=q^{(i-j)(t-j)}\cdot \gaussmnum{i}{j}{q}\cdot \gaussmnum{k-i}{t-j}{q}$ 
  for all $1\le j\le \min\!\left\{t,v_2\right\}$ and $\max\{d/2,j\}\le i\le \min\{k,d/2-1+j\}$.  
  For non-trivial parameters we have
  $$
    B_q(v_1,v_2,d;k)\le \sum_{i=\tfrac{d}{2}}^{\min\{k,v_2\}} a_i,
  $$
  where the $a_i$ are non-negative integers satisfying the constraints
  \begin{equation}
     \label{ie_t_subspace}
     \sum_{i=\max\{d/2,j\}}^{\min\{k,d/2-1+j\}} b(i,j)\cdot a_i \le q^{(v_2-j)(t-j)}\cdot \gaussmnum{v_2}{j}{q}\cdot \gaussmnum{v_1-v_2}{t-j}{q}
  \end{equation}
  for each $1\le j\le \min\!\left\{t,v_2\right\}$ and
  \begin{equation}
    \label{ie_tail}
    \sum_{i=h}^{\min\{k,v_2\}} a_i \le A_q(v_2,2(h-t+1);h)
  \end{equation}
  for all $\max\{t,d/2\}\le h\le\min\{k,v_2\}$.  
\end{Lemma}
\begin{proof}  
   Let $V:=\mathbb{F}_q^{v_1}$, $W$ be a $v_2$-space in $V$, and $\mathcal{C}$ be a set of $k$-spaces in $V$ that intersect $W$ in 
  dimension at least $\tfrac{d}{2}$. By $a_i$ we denote the number of elements in $\mathcal{C}$ that have an intersection of dimension 
  exactly $i$ with $W$, so that $\#\mathcal{C}=\sum_{i=d/2}^{\min\{k,v_2\}} a_i$. 
  We note that every $t$-space is contained in at most one element from $\mathcal{C}$.
  
  Now, consider a codeword $U\in \mathcal{C}$ with intersection dimension $i=\dim(U\cap W)$. The $\gaussmnum{k}{t}{q}$ $t$-spaces $T$ contained in $U$ 
  can be distinguished by their intersection dimension $j=\dim(T\cap W)$. For each $\max\{1,t-k+i\}\le j\le \min\{i,t\}$ there are exactly 
  $q^{(i-j)(t-j)}\cdot \gaussmnum{i}{j}{q}\cdot \gaussmnum{k-i}{t-j}{q}$ such $t$-spaces, see Lemma~\ref{lemma_count_subspaces}. Since the number of $t$-spaces 
  in $V$ that intersect $W$ in dimension $j$ is given by $q^{(v_2-j)(t-j)}\cdot \gaussmnum{v_2}{j}{q}\cdot \gaussmnum{v_1-v_2}{t-j}{q}$ we obtain (\ref{ie_t_subspace}).
  
  Next we set $\mathcal{C}'=\{U\cap W\,:\, U\in\mathcal{C},\dim(U\cap W)\ge h\}$, so that $\#\mathcal{C}'=\sum_{i=h}^{\min\{k,v_2\}} a_i$. Now 
  let $\mathcal{C}''$ arise from $\mathcal{C}'$ by choosing an arbitrary $h$-subspace from each $U'\in\mathcal{C}'$ as codeword $U''\in\mathcal{C}''$.
  By construction we have $\dim(A''\cap B'')\le t-1=k-d/2$ for each pair of different codewords $A'',B''\in\mathcal{C}''$, so that $d(A'',B'')\ge 2(h-t+1)\ge 2$ 
  and $\#\mathcal{C}'=\#\mathcal{C}''$. Thus $\mathcal{C}''$ is a $(v_2,\#\mathcal{C},2(h-t+1);h)_q$ code and we obtain (\ref{ie_tail}).         
\end{proof}

\begin{Corollary}
  \label{cor_upper_bound}
  For non-trivial parameters we have $B_q(v_1,v_2,d;k)\le$
  $$
    A_q(v_2, (\Lambda+1) d-2k;\Lambda d/2)+\sum_{l=1}^{\Lambda-1} q^{(v_2-ld/2)(k-(l+1)d/2+1)}\cdot \gaussmnum{v_2}{ld/2}{q}\cdot \gaussmnum{v_1-v_2}{k-(l+1)d/2+1}{q} / \gaussmnum{k-ld/2}{d/2-1}{q},
  $$  
  where $\Lambda:=\left\lfloor 2k/d\right\rfloor$.
\end{Corollary}
\begin{proof}
  We apply Lemma~\ref{lemma_lp_upper_bound} and use the corresponding notation, i.e., we will upper bound $\sum_{i=\tfrac{d}{2}}^{\min\{k,v_2\}} a_i$.

  For $k<d$ we have $d/2\ge k-d/2+1=t$ so that we can apply Inequality~(\ref{ie_tail}) with $h=d/2$ to conclude 
  the proposed upper bound for $\Lambda=1$.
  
  In the following we assume $k\ge d$, i.e., $\Lambda\ge 2$. From Equation~(\ref{eq_qbin_plus_1}) and  Equation~(\ref{eq_qbin_minus_1}) we conclude 
  $$
    \frac{b(i+1,j)}{b(i,j)} = q^{(t-j)}\cdot \frac{q^{i+1}-1}{q^{i-j+1}-1} \cdot \frac{q^{k-i-(t-j)}-1}{q^{k-i}-1} 
    =\frac{q^{k + 1} - q^{t +i -j+ 1} - q^{k - i} + q^{t - j}}{q^{k+1-j} - q^{i - j + 1} - q^{k - i} + 1 }.
  $$
  Using $i\le \left(d/2-1+j\right)-1=k-t+j-1$ and $i\le k-1$ we obtain
  $$
  \frac{b(i+1,j)}{b(i,j)} \ge 
  \frac{q^{k + 1} - q^{k} - q^{k - i} + q^{t - j}}{q^{k+1-j} - q^{k-j} - q^{k - i} + 1 } 
  \overset{j\ge 1}{\underset{t-j\ge 0}{\ge}} \frac{q^{k + 1} - q^{k} - q^{k - i} + 1}{q^{k} - q^{k-1} - q^{k - i} + 1 }\ge 1,
  $$  
  i.e., the sequence $\left(b_{i,j}\right)_i$ is weakly monotonic increasing. 
  Thus, using $a_i\ge 0$, we conclude 
  \begin{equation}
    \label{ie_head_1}
    \sum_{i=ld/2}^{(l+1)d/2-1} b(ld/2,ld/2)\cdot a_i \le q^{(v_2-ld/2)(t-ld/2)}\cdot \gaussmnum{v_2}{ld/2}{q}\cdot \gaussmnum{v_1-v_2}{t-ld/2}{q}
  \end{equation}
  from Inequality~(\ref{ie_t_subspace}) with $j=ld/2$, where $1\le l< \Lambda$. Here we note that $\max\{d/2,j\}=ld/2$, due to $l\ge 1$, and 
  $\min\{k,d/2-1+j\}=(l+1)d/2-1$, due to  $l\le \left\lfloor 2k/d\right\rfloor-1$. Dividing Inequality~(\ref{ie_head_1}) by $b(ld/2,ld/2)$ 
  gives
  \begin{equation}
    \label{ie_head_2}
    \sum_{i=ld/2}^{(l+1)d/2-1} a_i \le q^{(v_2-ld/2)(t-ld/2)}\cdot \gaussmnum{v_2}{ld/2}{q}\cdot \gaussmnum{v_1-v_2}{t-ld/2}{q} / \gaussmnum{k-ld/2}{d/2-1}{q}
  \end{equation} 
  using $\gaussmnum{k-ld/2}{t-ld/2}{q}=\gaussmnum{k-ld/2}{d/2-1}{q}$. Since Inequality~(\ref{ie_tail}) with $h=\Lambda d/2$ 
  gives 
  \begin{equation}
    \label{ie_special_tail}
    \sum_{i=\Lambda d/2}^{\min\{k,v_2\}} a_i \le A_q(v_2, (\Lambda+1) d-2k;\Lambda d/2)
  \end{equation}
  we can add the right hand side of Inequality~(\ref{ie_special_tail}) to the sum over the right hand side of Inequality~(\ref{ie_head_2}) for $1\le l<\Lambda$ to 
  conclude the proposed upper bound.   
\end{proof}

Actually, Corollary~\ref{cor_upper_bound} is a generalization of the known upper bounds for {\cdc}s that contain a lifted {\MRD} code as a subcode:
\begin{Proposition}
  \label{prop_mrd_bound}
  Let $v$, $k$, and $d/2$ be positive integers with $d\le 2k\le v$ and $\mathcal{C}$ be a $(v,\star,d;k)_q$ code that contains a lifted {\MRD} code $\mathcal{C}'$ 
  of cardinality $q^{(v-k)\cdot(k-d/2+1)}$ as a subcode. Then, we have $\#C\le q^{(v-k)\cdot(k-d/2+1)} +B_q(v,v-k,d;k)$.  
\end{Proposition}  
\begin{proof}
  Let $W$ be the $(v-k)$-space that is disjoint from all codewords of $\mathcal{C}'$. From e.g.\ \cite[Lemma 4]{MR3015712} we know that every $(k-d/2+1)$-space that is  
  disjoint to $W$ is contained in a codeword from $\mathcal{C}'$. Thus, the codewords in $\mathcal{C}\backslash\mathcal{C}'$ have to intersect $W$ in 
  dimension at least $d/2$.
\end{proof}

Applying Corollary~\ref{cor_upper_bound} gives:
\begin{Corollary}
  \label{cor_mrd_bound}
  Let $v$, $k$, and $d/2$ be positive integers with $d\le 2k\le v$ and $\mathcal{C}$ be a $(v,\star,d;k)_q$ code that contains a lifted {\MRD} code $\mathcal{C}'$ 
  of cardinality $q^{(v-k)\cdot(k-d/2+1)}$ as a subcode. Then,
  \begin{eqnarray*}
    \# C&\le& q^{(v-k)\cdot(k-d/2+1)} + A_q(v-k, (\Lambda+1) d-2k;\Lambda d/2)\\ 
    && +\sum_{l=1}^{\Lambda-1} q^{(v-k-ld/2)(k-(l+1)d/2+1)}\cdot \gaussmnum{v-k}{ld/2}{q}\cdot \gaussmnum{k}{(l+1)d/2-1}{q} / \gaussmnum{k-ld/2}{d/2-1}{q},
  \end{eqnarray*}
  where $\Lambda:=\left\lfloor 2k/d\right\rfloor$.
\end{Corollary}
  
The cases $\Lambda\le 2$, i.e.\ $k<3d/2$, cover \cite[Theorem 1]{heinlein2019new} as well as its predecessors \cite[Theorem 10]{MR3015712} and \cite[Theorem 11]{MR3015712}. 
For $\Lambda\ge 3$, i.e.\ $k\ge 3d/2$, Corollary~\ref{cor_mrd_bound} gives new upper bounds. As an example we consider the binary case $(v,d;k)_q=(12,4;6)_2$, where a {\cdc} 
$\mathcal{C}$ that contains a lifted {\MRD} code has to satisfy
$$
  \#\mathcal{C} \le 1\,321\,780\,637,
$$
noting the best known general bounds
$$
  1\,212\,491\,081 \le A_2(12,4;6)\le 1\,816\,333\,805. 
$$

Next we show that the upper bound of Corollary~\ref{cor_upper_bound} for $B_q(v_1,v_2,d;k)$ is tight for $k<d$, i.e., those cases where the bound does not depend on $v_1$, 
provided that $v_1$ is sufficiently large.
\begin{Proposition}
\label{prop_asym_k_smaller_d}
For non-trivial parameters we have 
$$
  B_q(v_1,v_2,d;k)=A_q(v_2, 2d-2k;d/2)
$$
if $k<d$ and $v_1\ge v_2k$.  
\end{Proposition}
\begin{proof}
  Due to Corollary~\ref{cor_upper_bound} it remains to construct a code $\mathcal{C}$ with cardinality $A_q(v_2, 2d-2k;d/2)$ that satisfies the conditions of Definition~\ref{def_B_q}.  
  To this end, let $W:=\mathbb{F}_q^{v_2}\le \mathbb{F}_q^{v_2}\times \mathbb{F}_q^{v_1-v_2}=:V$ and let $\mathcal{F}$ be a $(v_2,N,2d-2k;d/2)_q$ code of maximal size in $W$, i.e., 
  $N=A_q(v_2,2d-2k;d/2)$. If $d=2$, then $k=1=d/2$, so that we can set $\mathcal{C}=\mathcal{F}$. In the following we assume $d\ge 4$ and set 
  $t=k-d/2$. Let $\mathcal{P}$ be a partial $t$-spread in $\mathbb{F}_q^{v_1-v_2}$ of cardinality $A_q(v_1-v_2,2t;t)$, so that
  $$
    \# \mathcal{P}\overset{(\ref{eq_asymptotic_explicit})}{\ge} q^{v_1-v_2-k+d/2}.
  $$
  Since
  $$
    \# \mathcal{F}=A_q(v_2, 2d-2k;d/2) \overset{(\ref{eq_asymptotic_explicit})}{\le} q\cdot q^{(v_2-d/2)(k-d/2+1)}, 
  $$
  we have $\#\mathcal{P}\ge \#\mathcal{F}$ if
  $$ 
    v_1-v_2-k+d/2 \,\,\ge\,\, 1+(v_2-d/2)(k-d/2+1),
  $$
  which is equivalent to 
  \begin{equation}
    \label{ie_asym_k_smaller_d}
    v_1\,\,\ge\,\, v_2+(v_2-d/2+1)(k-d/2+1).
  \end{equation}
  Since $d\ge 4$ and $k\ge 1$ the right hand side of (\ref{ie_asym_k_smaller_d}) is at most $v_2k$, i.e., we indeed have $\#\mathcal{P}\ge \#\mathcal{F}$. So, for each 
  element $U\in\mathcal{F}$ we can choose a different element $f(U)\in \mathcal{P}$ and set $\mathcal{C}=\{U\times f(U)\mid U\in \mathcal{F}\}$, 
  which has the desired properties of Definition~\ref{def_B_q} by construction.     
\end{proof}

Actually our estimate in the proof of Proposition~\ref{prop_asym_k_smaller_d} is rather rough and we expect that we have  $B_q(v_1,v_2,d;k)=A_q(v_2, 2d-2k;d/2)$ 
for $k<d$ and much smaller values of $v_1$. To this end we propose the following heuristic algorithm:

  \begin{algorithmic}
  \STATE{let $W=\mathbb{F}_q^{v_2}\le \mathbb{F}_q^{v_1}=:V$, $\mathcal{F}$ be a $(v_2,\star,2d-2k;d/2)_q$ code, and $t:=k-d/2$}
  \FOR {each $t$-space $T\le W$}
    \STATE{$\mathcal{F}_T\leftarrow\left\{U\in\mathcal{F}\mid T\le U\right\}$} 
    \STATE{$\mathcal{E}_T\leftarrow\{(k-d/2)-\text{space in }V\text{ with trivial intersection with }W\}$} 
  \ENDFOR
  \FOR{ each $U\in\mathcal{F}$}
    \STATE{$f(U)\leftarrow \emptyset$}
  \ENDFOR
  \FOR {each $t$-space $T\le W$}
    \FOR {$U\in\mathcal{F}_T$ with $f(U)=\emptyset$}
      \IF{$\mathcal{E}_T=\emptyset$}
        \STATE{$f(U)=\varnothing$}
      \ELSE
        \STATE{choose $Q\in \mathcal{E}_T$ and set $C=\langle U,Q\rangle$}
        \IF{$d(C,f(U'))\ge d$ for all $U'\in \mathcal{F}$ with $f(U')\notin\{\emptyset,\varnothing\}$}
         \STATE{set $f(U)=C$}
        \ELSE          
          \STATE{$\mathcal{E}_T\leftarrow \mathcal{E}_{T}\backslash\{Q\}$}
        \ENDIF
      \ENDIF
    \ENDFOR
  \ENDFOR
  \RETURN {$\mathcal{M}=\{f(U)\mid U\in\mathcal{F},f(U)\neq\emptyset,f(U)\neq\varnothing\}$}
  \end{algorithmic}

Here we use the symbol $\emptyset$ for the empty set and the symbol $\varnothing$ in order to denote the fact that we cannot find an extension $f(U)$ for $U$.   
By construction $\mathcal{M}$ is a $(v_1,\#\mathcal{M},d;k)_q$ code with $\#\mathcal{M}\le A_q(v_2, 2d-2k;d/2)$ such that every codeword intersects $W$ with dimension 
at least $d/2$, i.e., the assumptions of Definition~\ref{def_B_q} are satisfied and we obtain the lower bound $B_q(v_1,v_2,d;k)\ge \#\mathcal{M}$. Of course the cardinality 
of the resulting {\cdc} $\mathcal{M}$ depends on the ordering of the choices in the for-loops. A number of trials leads to:
\begin{Conjecture} 
  \label{conj_d_2k_2}
  If $v_1\ge v_2+2\ge k+1$ and $k\ge 3$, then $$B_q(v_1,v_2,2k-2;k)= A_q(v_2,2k-4;k-1).$$
\end{Conjecture}

If Conjecture~\ref{conj_d_2k_2} can be proven to be true, this would have several implications, see Appendix~\ref{appendix_implications_conjecture}. Indeed, 
several of the currently best known lower bounds for $A_q(v,d;k)$ could be improved. Next we give a few exemplary constructions for lower bounds for $B_q(v_1,v_2,d;k)$, 
not covered by Conjecture~\ref{conj_d_2k_2}, which yield strict improvements for $A_q(v,d;k)$ (and $q\ge 3$). 
\begin{Proposition}   
  \label{prop_last}
  We have 
  $$A_q(12,4;4)\ge q^{24}\!+\!q^{20}\!+\!q^{19}\!+\!3q^{18}\!+\!2q^{17}\!+\!3q^{16}\!+\!q^{15}\!+\!q^{14}\!+\!q^{12}\!+\!q^{10}\!+\!2q^8\!+\!2q^6\!+\!2q^4\!+\!q^2\!+\!1$$ and
  $$A_q(13,4;4)\ge q^{27}\!+\!q^{23}\!+\!q^{22}\!+\!3q^{21}\!+\!2q^{20}\!+\!3q^{19}\!+\!q^{18}\!+\!q^{17}\!+\!q^{15}\!+\!q^{12}\!+\!q^{10}\!+\!q^9\!+\!q^8\!+\!q^7\!+\!q^6\!+\!q^5\!+\!q^3.$$
\end{Proposition}
\begin{Proof}
  It has been proved several times that $$A_q(8,4;4)\ge q^{12}+q^8+q^7+3q^6+2q^5+3q^4+q^3+q^2+1,$$ see e.g.\ \cite[Theorem 18, Remark 6]{MR3015712}. 
  Using Theorem~\ref{main_thm} with $m=8$ gives $$A_q(12,4;4)\ge A_q(8,4;4)\cdot q^{12}+B_q(12,4,4;4)$$ and 
  $$A_q(13,4;4)\ge A_q(8,4;4)\cdot q^{15}+B_q(13,5,4;4).$$ 
  Next we want to give a constructive lower bound for $B_q(12,4,4;4)$ and $B_q(13,5,4;4)$, respectively.  
  
  We start with a lower bound for $B_q(12,4,4;4)$. Let $W$ be an arbitrary but fix solid, i.e., a $4$-space, in $V=\F_q^{12}$. Let $\mathcal{P}_W$ be a line spread of cardinality $A_q(4,4;2)=q^2+1$ of $W$.   
  For each line $L$ in $\mathcal{P}_W$ there exist $\alpha:=q^8 +q^6+q^4+q^2$ solids in $V$ that intersect $W$ in $L$ and have pairwise subspace distance $d=4$, as we will show 
  subsequently. To this end, let $L$ be an element of $\mathcal{P}_W$ and consider a line spread $\mathcal{P}(L)$ of $V/L\cong \F_q^{10}$. For each arbitrary representative $L_i$ of the 
  $A_q(10,4;2)=q^8 +q^6+q^4+q^2+1=\alpha+1$ elements of $\mathcal{P}(L)$ in $V$ we can construct the solid $\langle L_i,L\rangle$. Note that different representatives of an 
  element of $\mathcal{P}(L)$ in $V$ lead to the same solid in $V$. By construction, these $\alpha+1$ solids have pairwise subspace distance $4$ and contain $L$. W.l.o.g.\ we 
  can choose the line spread $\mathcal{P}(L)$ and the numbering of the $L_i$ such that $\langle L_1,L\rangle=W$. Then, we choose only the $\alpha$ solids  
  $\langle L_2,L\rangle,\dots, \langle L_{\alpha+1},L\rangle$, which have to intersect $W$ in dimension 
  exactly $2$, since they have subspace distance $2$ to $\langle L_1,L\rangle=W$. Thus, for each line $L$ in $\mathcal{P}_W$ there exist $\alpha$ solids in $V$ that intersect $W$ in  
  $L$ and have pairwise subspace distance $d=4$,    
  
  Now we apply this construction for every line $L$ of the line spread $\mathcal{P}_W$. Additionally adding $W$ itself as a codeword gives 
  a set $\mathcal{C}'$ of $(q^2+1)\cdot \alpha+1=(q^2+1)(q^8 +q^6+q^4+q^2)+1$ solids that intersect $W$ with dimension at least $2$.  
  Finally, we check that for different $L,L'\in \mathcal{P}_W$ and different $L_j$, $L_i$ as defined above, we have $\dim(\langle L,L_i \rangle\cap\langle L,L_j\rangle)=2$,  
  $\dim(\langle L,L_i \rangle\cap\langle L',L_i\rangle)=2$, $\dim(\langle L,L_i \rangle\cap\langle L',L_j\rangle)\le 2$, and $\dim(\langle L,L_i \rangle\cap W)\le 2$, so that the 
  minimum subspace distance of $\mathcal{C}'$ is $4$. Thus, $B_q(12,4,4;4)\ge (q^2+1)\cdot (q^8 +q^6+q^4+q^2)+1$.
  
  \medskip  
  
  For a lower bound for $B_q(13,5,4;4)$ we set $V=\F_q^{13}$ and choose a $5$-space $W$ in $V$, which admits a partial line spread $\mathcal{P}_W$ of cardinality $A_q(5,4;2)=q^3+1$. 
  Again, we extend each line $L$ in $\mathcal{P}_W$ to several solids in $V$ intersecting $W$ only in $L$ and having pairwise subspace distance $4$. To that end, given a line 
  $L$ in $\mathcal{P}_W$ we consider a partial line spread $\mathcal{P}(L)$ of $V/L\cong \F_q^{11}$ that is disjoint from a plane $\pi$, where we assume  that $L$ and a representative 
  of $\pi$ in $V$ generate $W$. The maximum size of this $\mathcal{P}(L)$ partial line spread is $A_q(11,4;2)-1=q^9 +q^7+q^5+q^3$, so that 
  $B_q(13,5,4;4)\ge (q^3+1)\cdot (q^9 +q^7+q^5+q^3)$ using a similar distance analysis as above. 
  (Again, we may add an additional solid contained in $W$ as a codeword.)   
\end{Proof}

We remark that the previously best known lower bound for $A_q(12,4;4)$ and $A_q(13,4;4)$ for all $q\ge 2$ is given by the improved linkage construction for $m=8$, 
i.e., 
\begin{eqnarray*}
A_q(12,4;4)&\ge& A_q(8,4;4)\cdot q^{12}+A_q(6,4;4) 
\,\,=\,\,A_q(8,4;4)\cdot q^{12}+A_q(6,4;2)\\ 
&\ge& q^{24}+q^{20}+q^{19}+3q^{18}+2q^{17}+3q^{16}+q^{15}+q^{14}+q^{12} +q^4+q^2+1
\end{eqnarray*} 
and 
\begin{eqnarray*}
A_q(13,4;4)&\ge& A_q(8,4;4)\cdot q^{15}+A_q(7,4;4) 
\,\,=\,\,A_q(8,4;4)\cdot q^{15}+A_q(7,4;2)\\ 
&\ge& q^{27}\!+\!q^{23}\!+\!q^{22}\!+\!3q^{21}\!+\!2q^{20}\!+\!3q^{19}\!+\!q^{18}\!+\!q^{17}\!+\!q^{15}+q^5+q^3+1,
\end{eqnarray*}
using $A_q(7,4;2)=q^5+q^3+1$. Very recently, the lower bounds for $A_q(12,4;4)$ and $A_q(13,4;4)$ were further improved in \cite[Theorem 5.4]{cossidente2019combining} 
surpassing the bound of Proposition~\ref{prop_last}.

Another case where Theorem~\ref{main_thm} yields a strict improvement is $A_q(16,6;5)$. Here the previously best known lower bound is obtained via 
the (improved) linkage construction with $m=11$, i.e., 
\begin{eqnarray*}
 A_q(16,6;5)&\ge& A_q(11,6;5)\cdot q^{15}+A_q(7,6;5) \\ &=&A_q(11,6;5)\cdot q^{15}+A_q(5,6;5)=A_q(11,6;5)\cdot q^{15}+1.
\end{eqnarray*} 
So, we get a strict improvement if $B(16,5,6;5)>1$, which is certainly true. E.g., in a $5$-space $W$ of $V=\F_q^{16}$ we can choose 
$\gaussmnum{5}{3}{q}=\gaussmnum{5}{2}{q}=q^6+q^5+2q^4+2q^3+2q^2+q+1$ different planes that pairwise intersect in a point, i.e., that have subspace distance $2$. In 
$V/W\cong\F_q^{11}$ we can choose a partial line spread of cardinality at least $\gaussmnum{5}{2}{q}<q^9<A_q(11,4;2)$, so that we can extend each of the planes by a disjoint 
line from the partial line spread to obtain $\gaussmnum{5}{2}{q}$ $5$-spaces 
with pairwise subspace distance $2+4=6$, i.e., $B(16,5,6;5)\ge \gaussmnum{5}{2}{q}=q^6+q^5+2q^4+2q^3+2q^2+q+1$ and
\begin{equation}
  A_q(16,6;5)\ge A_q(11,6;5)\cdot q^{15}+\gaussmnum{5}{2}{q}.
\end{equation}

\section{Conclusion}
We have generalized the linkage construction, which is one of the two most successful construction strategies for {\cdcs} with large size, in our main 
theorem~\ref{main_thm}. This comes at the cost of introducing the new quantity $B_q(v_1,v_2,d;k)$. We have presented upper bounds for $B_q(v_1,v_2,d;k)$ 
that generalized the best known upper bounds for {\cdc}s that contain a lifted {\MRD} code. With respect to lower bounds for $B_q(v_1,v_2,d;k)$ we give 
a few, in the field size $q$, parametric examples. 
In \cite{XC} lifted {\MRD} codes have been augmented by adding an additional {\cdc} $\mathcal{C}$, which is constructed via rank metric codes with bounds on the rank 
of the matrices. It turns out that $\mathcal{C}$ corresponds to a {\cdc} that matches the requirements of Definition~\ref{def_B_q}, i.e., the results of 
\cite{XC} can be reformulated as lower bounds for $B_q(v_1,v_2,d;k)$. This is remarked explicitly in \cite{DH}, see also \cite[Lemma 4.1]{K1}.  

The study of lower and upper bounds for $B_q(v_1,v_2,d;k)$ might be a promising research direction on its own. We remark that the linkage construction can also be generalized 
to mixed dimension codes, i.e., sets of codewords from $P(V)$ with arbitrary dimensions. However, other known constructions are superior to that approach.    

\section*{Acknowledgment}
The author would like to thank the anonymous referees of for their careful reading, very helpful remarks, suggestions, and patience, which significantly 
improved the presentation of this paper. One referee even improved the initial statement of Proposition~\ref{prop_d_4_k_3} and  
Proposition~\ref{prop_last}.

\appendix
\section{Implications of Conjecture~\ref{conj_d_2k_2}}
\label{appendix_implications_conjecture}
In this appendix we want to collect a few implications of Conjecture~\ref{conj_d_2k_2}. We will see that the truth of Conjecture~\ref{conj_d_2k_2} would 
unify and generalize several known lower bounds from the literature. In an earlier version we proposed an algorithm that would 
prove Conjecture~\ref{conj_d_2k_2} but which unfortunately was flawed.  

\begin{Theorem}
  If Conjecture~\ref{conj_d_2k_2} is true, then we have
  \label{thm_d_2k_2}
  $$A_q(v,2k-2;k)\ge A_q(m,2k-2;k)\cdot q^{2(v-m)}+A_q(v-m,2k-4;k-1)$$ for $m\ge k\ge 3$.
\end{Theorem}             
\begin{proof}
  Apply Theorem~\ref{main_thm}.
\end{proof}

Let us consider an example. For $q\ge 3$ the best known lower bound for $A_q(10,4;3)$ is obtained by the linkage construction, i.e., Inequality~(\ref{ie_linkage_construction}), 
with $m=7$. More precisely, we have $A_q(7,4;3)\ge q^8+q^5+q^4+q^2-q$ for every prime power $q$ \cite[Theorem 4]{MR3444245}. (For $q=2,3$ better constructions are known 
\cite{paper_333,MR3444245}.) Lifting gives an extra factor of $q^6$ and linkage as well as improved linkage, i.e., Inequality~(\ref{ie_linkage_construction}) and 
Inequality~(\ref{ie_linkage_construction_improved}), give only one additional 
codeword, so that 
$$
A_q(10,4;3)\ge \left(q^8+q^5+q^4+q^2-q\right)\cdot q^6\,+\,1=
q^{14}+q^{11}+q^{10}+q^8-q^7+1.
$$
Applying Theorem~\ref{thm_d_2k_2} with $m=7$ gives the better lower bound $$A_q(10,4;3)\ge q^{14}+q^{11}+q^{10}+q^8-q^7+q^2+q+1,$$ since $A_q(3,2;2)=A_q(3,2;1)=q^2+q+1$. We remark 
that the lower bound $B_q(v,3,4;3)\ge q^2+q+1$ of Conjecture~\ref{conj_d_2k_2}, is indeed attained with equality for all $v\ge 3$.   

We can also obtain other constructions from the literature as special cases, see the subsequent discussion.
\begin{Corollary}$\,$\\[-3mm]
  \label{cor_d_2k_2}
  \begin{itemize}
    \item[(a)] $A_q(v,2k-2;k)\ge q^{2(v-k)}+A_q(v-k,2k-4;k-1)$ for $k\ge 3$. 
    \item[(b)] $A_q(3k-3,2k-2;k)\ge q^{4k-6}+q^{k-1}+1$ for $k\ge 3$.
  \end{itemize}
\end{Corollary}  
\begin{Proof}
  For part~(a) we apply Theorem~\ref{thm_d_2k_2} with $m=k$. Specializing to $v=3k-3$ and using $A_q(2k-3,2k-4;k-1)=A_q(2k-3,2k-4;k-2)=q^{k-1}+1$, see \cite[Theorem 4.2]{MR0404010},  
  then gives part~(b).
\end{Proof}

With the extra condition $q^2+q+1\ge 2\left\lfloor v/2\right\rfloor-3$ part~(a) is equivalent to \cite[Theorem 16, Construction 1]{MR3015712}.
For e.g.\ $v=8$ and $k=3$ the corresponding lower bound $A_q(8,4;3)\ge q^{10}+\gaussmnum{5}{2}{q}=q^{10}+q^6+q^5+2q^4+2q^3+2q^2+q+1$ is indeed the 
best known lower bound for $q\ge 3$. Part~(b) matches the coset construction \cite[Theorem 11]{heinlein2017coset}, which is valid for $k\ge 4$. 
Moreover, this explicit lower bound matches the best known lower bound for $k=4,5,6,7$ and $q\ge 2$, where it is also achieved by the Echelon-Ferrers 
construction. 

For $k=3$ the following proposition strictly improves the previously best known lower bounds for $q\ge 4$ and $t\ge 1$.

\begin{Proposition}
  \label{prop_d_4_k_3} 
  Assuming that Conjecture~\ref{conj_d_2k_2} is true, for $t\ge 0$ we have
  \begin{eqnarray*}
    A_q(7+3t,4;3)&\ge& \left(q^8+q^5+q^4+q^2-q\right)\cdot q^{6t}+\gaussmnum{3t}{2}{q},\\
    A_q(8+3t,4;3)&\ge& \left(q^{10}+q^6+q^5+2q^4+2q^3+2q^2+q+1\right)\cdot q^{6t}+\gaussmnum{3t}{2}{q},\text{ and}\\ 
    A_q(9+3t,4;3)&\ge& \left(q^{12}+2q^8+2q^7+q^6+2q^5+2q^4-2q^2-2q+1\right)\cdot q^{6t}+\gaussmnum{3t}{2}{q}.\\
  \end{eqnarray*}
\end{Proposition}
\begin{Proof}
  For $t=0$ we have $A_q(7,4;3)\ge q^8+q^5+q^4+q^2-q$ \cite{MR3444245}, $A_q(8,4;3)\ge q^{10}+q^6+q^5+2q^4+2q^3+2q^2+q+1$, and 
  $A_q(9,4;3)\ge q^{12}+2q^8+2q^7+q^6+2q^5+2q^4-2q^2-2q+1$ \cite[Corollary 4]{K1}. 
  For $t\ge 1$ let $a\in\{7,8,9\}$, 
  $v=a+3t$, and $m=v-3t$, i.e., $m=a$. Applying Theorem~\ref{thm_d_2k_2} with $k=3$ gives the stated formulas.  
\end{Proof}

The last two parametric inequalities would also strictly improve the best known lower bounds for $q=3$ and $t\ge 1$. Also for $k>3$ 
strict improvements can be concluded from Theorem~\ref{thm_d_2k_2}.
\begin{Proposition}
  Assuming that Conjecture~\ref{conj_d_2k_2} is true, we have 
  \begin{eqnarray*}
    A_q(10,6;4)&\ge& q^{12} +q^6+2q^2+2q+1,\\ 
    A_q(13,6;4)&\ge& q^{18} +q^{12}+2q^8+2q^7+q^6+2q^5+2q^4-2q^2-2q+1,\text{ and}\\
    A_q(14,6;4)&\ge& q^{20}+q^{14}+q^{11}+q^{10}+q^8-q^7+q^2+q+1. 
  \end{eqnarray*} 
\end{Proposition}
\begin{Proof}
  Since $A_q(6,4;3)\ge q^6+2q^2+2q+1$, see e.g.\ \cite[Theorem 2]{MR3329980}, we conclude $A_q(10,6;4) \ge q^{12} +q^6+2q^2+2q+1$
  from Corollary~\ref{cor_d_2k_2}.(a), where we set $k=4$. 
  Using Proposition~\ref{prop_d_4_k_3} we conclude the second and the third lower bound
  from Corollary~\ref{cor_d_2k_2}.(a) with $k=4$. 
\end{Proof}    


\end{document}